\documentclass[12pt]{amsart}
\usepackage[lite, alphabetic, nobysame]{amsrefs}
\usepackage{amsmath,mathtools}
\usepackage{amssymb}
\usepackage{srcltx}
\usepackage{xcolor}

\newtheorem{theorem}{Theorem}
\newtheorem{lemma}{Lemma}

\newtheorem*{claim}{Claim}
\newtheorem{fact}{Fact}

\theoremstyle{remark}

\title[The random poset does not admit a generic pair]{The automorphism group of the random poset does not admit a generic pair}

\author[A. Kwiatkowska]{Aleksandra Kwiatkowska}
\address{Institut f\"ur Mathematische Logik und Grundlagenforschung, Westfalische Wilhelms-Universit\"at M\"unster, Einsteinstr. 62, 48149 M\"unster, Germany  And Instytut Matematyczny, Uniwersytet Wrocławski, pl. Grunwaldzki 2/4, 50-384 Wrocław, Poland}
\email{kwiatkoa@uni-muenster.de}

\author[A. Panagiotopoulos]{Aristotelis Panagiotopoulos}
\address{Institut f\"ur Mathematische Logik und Grundlagenforschung, Westfalische Wilhelms-Universit\"at M\"unster, Einsteinstr. 62, 48149 M\"unster, Germany 
And JG
}
\email{aristotelis.panagiotopoulos@gmail.com}

\thanks{We would like to thank A.S. Kechris for his feedback on an earlier draft of this paper.}
\thanks{Funded by the Deutsche Forschungsgemeinschaft (DFG, German Research Foundation) under Germany's Excellence Strategy EXC 2044--390685587, Mathematics M\"unster: Dynamics--Geometry--Structure
and by CRC 1442 Geometry: Deformations and Rigidity.}
\thanks{A.K. additionally was  supported by Narodowe Centrum Nauki grant 2016/23/D/ST1/01097.}



\begin{document}
\maketitle

\begin{abstract}
We show that the conjugacy class of every pair of automoprhisms of the random poset is meager. This answers a question of Truss; see also Kuske-Truss. 

\noindent EDIT: Work in progress, at the moment there is a gap in the proof of Theorem 2.
\end{abstract}

\section{Introduction}

Some of the most well-studied properties of Polish groups concern the interactions between the topological and algebraic structure of such groups. Indeed, properties such as
the \emph{small index property},  the \emph{automatic continuity}, and the \emph{Bergman property} are  often used to evaluate how much of the  topology and the dynamical properties of  a Polish group $G$  can be recovered from its pure algebraic structure.  While most of these properties rarely occur in the realm of locally compact Polish groups, they tend to be a typical feature of automorphism groups $\mathrm{Aut}(M)$ of classical ultrahomogeneous  structures $M$, such as the \emph{random graph} and the linear ordering $(\mathbb{Q},<)$.

That being said, there are still two structures among these classical examples for which it  remains open whether their automorphism groups satisfy any of the above properties---namely, the \emph{random poset} $(P,<)$ and the \emph{random tournament} $(T,\rightarrow)$.
The main obstacle here is that the only known techniques we have for establishing such properties for an automorphism group  $\mathrm{Aut}(M)$  use either the existence of elements $g\in\mathrm{Aut}(M)$ with ``small" \emph{support} $\mathrm{supp}(g):=\{a\in M \mid g(a)\neq a\}$, see~\cite{B,RS}; or  require $\mathrm{Aut}(M)$ to satisfy a relatively stronger property known as \emph{ample generics} \cite{KR}. However, both   $\mathrm{Aut}(P,<)$ and $\mathrm{Aut}(T,\rightarrow)$ fail to have elements with ``small" support (see~\cite[Lemma 3.1]{GCR} and \cite[Lemma 2.9]{MT})
and for both structures it has been a long-standing open question whether their automorphism groups have ample generics (for $\mathrm{Aut}(P,<)$, see~\cite{KT,T}; for  $\mathrm{Aut}(T,\rightarrow)$, see \cite{HL} and \cite[Proposition~6.3.1]{S}). In this paper we settle the later question for $(P,<)$  in the negative.  As a consequence,  any potential proof that   $\mathrm{Aut}(P,<)$  has, say, the automatic continuity property, would truly have to involve   novel techniques.

\subsection*{Defintions and notation} A {\bf poset} (partially ordered set) is a structure $(N,<_N)$ so that for all $a,b,c\in N$ we have that:  $(a<_N b)\wedge(b<_Nc)$ implies $a<_Nc$; and $a\not<_Na$. We write $a\perp b$ whenever $a\neq b$ and both $a\not<_Nb$, $b\not<_Na$ hold. 
%Whenever there is no ambiguity regarding the poset we are referring to, we will be  denoting $<_N$ simply by $<$.
 An {\bf embedding of posets} $f\colon (A,<_A)\to (B,<_B)$ is any injective function $f\colon A\to B$ so that $(a<_Aa')\iff (f(a)<_B f(a'))$. 
The {\bf random poset} $\mathcal{P}$  is the unique  up to isomorphism countable poset $(P,<_P)$ so that: 
\begin{enumerate}
\item every finite poset embeds in $\mathcal{P}$;
\item if $f,g \colon (A,<_A) \to \mathcal{P}$ are two embeddings  and $A$ is finite, then  $\varphi\circ f=g$ for some automorphism  $\varphi\in\mathrm{Aut}(\mathcal{P})$ of $\mathcal{P}$.
\end{enumerate}

Below we view the automorphism group  $\mathrm{Aut}(\mathcal{P})$ of $\mathcal{P}$ as a Polish group endowed with the pointwise convergence topology. 
Let $G$ be a Polish  group. We say that $f\in G$ is {\bf generic} if the conjugacy class $\{gfg^{-1}\colon g\in G\}$ of $f$
is comeager. Furthermore, $(f_1,f_2)\in G^2$ is a {\bf generic pair} if the conjugacy class 
$\{(gf_1g^{-1}, gf_2g^{-1})\colon g\in G\}$ of $(f_1,f_2)$ is comeager in $G^2$.
More generally, we say that  $G$ has {\bf ample generics} if for every $n\geq 1$ there exists $(f_1,\ldots, f_n)$
whose diagonal conjugacy class $\{(gf_1g^{-1}, \ldots, gf_ng^{-1})\colon g\in G\}$ is comeager in $G^n$.

Kuske-Truss \cite{KT} showed that  $\mathrm{Aut}(\mathcal{P})$ has a generic automorphism. In  Section 3  of the same paper they remark that the next principal goal should be to establish that  ${\rm Aut}(\mathcal{P})$ has ample generics and to use this as a mean to prove that ${\rm Aut}(\mathcal{P})$  has the small index property. They also point out that merely the existence of a generic pair in ${\rm Aut}(\mathcal{P})$  could be used to streamline the proof of the fact that  ${\rm Aut}(\mathcal{P})$ is simple, see  \cite{GCR}. The question of whether  ${\rm Aut}(P)$  has a generic pair reappears in \cite[Question 3]{KT}. 
The following is the main result of this article.

\begin{theorem}\label{T:intro}
The automorphism group $\mathrm{Aut}(\mathcal{P})$ of the random poset $\mathcal{P}$ does not have a generic pair.
 In fact, for every $(f_1,f_2)\in \mathrm{Aut}(\mathcal{P})^2$ the diagonal conjugacy class $\{(gf_1g^{-1}, gf_2g^{-1})\colon g\in \mathrm{Aut}(\mathcal{P})\}$ of $(f_1,f_2)$   is meager  in $\mathrm{Aut}(\mathcal{P})^2$.
\end{theorem}

%Moreover, they are interested whether the group ${\rm Aut}(P)$ has ample generics, in particular, whether it has the small index property (see the discussion in Section 3 of \cite{KT}). By the work of Kechris-Rosendal  \cite[Theorem 6.9]{KR}  the ample generics property would imply the small index property. Our result implies that  ${\rm Aut}(P)$ does not have ample generics. Truss \cite{T} explicitly asks whether  ${\rm Aut}(P)$ has a generic pair (\S 4, Question 3). The result about the small index property remains open. See also Question 5.2.7(ii) in the survey of Macpherson \cite{M}.

\section{Proof of Theorem \ref{T:intro} }

We  start by reformulating the statement of Theorem \ref{T:intro} in combinatorial terms. For this we will make use of results from \cite{KR}.

A  {\bf partial automorphism} is a triple $(A,<_A,f_A)$, where $(A,<_A)$ is a finite poset and  $f_A\colon A \rightharpoonup A$ is a partial map, which induces an isomorphism between $\mathrm{dom}(f_A)$ and   $\mathrm{rng}(f_A)$. In what follows we will identify such $f_A$ with its graph, i.e., $f_A\subseteq A^2$.
 By an {\bf embedding $\alpha \colon (A,<_A,f_A)\to (B,<_B,f_B)$ of partial  automorphisms} we mean any embedding $\alpha \colon (A,<_A) \to (B,<_B)$ of posets  so that  $\mathrm{rng}(\alpha\upharpoonright \mathrm{dom}(f_A)) \subseteq \mathrm{dom}(f_B)$ and  $f_B \circ \alpha=\alpha\circ f_A$, for every $a\in\mathrm{dom}(f_A)$. A {\bf pair of partial automorphisms} is a quadruple $(A,<,f_A,g_A)$, so that both $(A,<_A,f_A)$ and $(A,<_A,g_A)$ are partial automorphisms.  By an {\bf embedding $\alpha \colon (A,<_A,f_A,g_A)\to (B,<_B,f_B,g_B)$ of pairs of partial  automorphisms} we mean any map $\alpha\colon A \to B$ which induces  embeddings $\alpha \colon (A,<_A,f_A)\to (B,<_B,f_B)$ and $\alpha \colon (A,<_A,g_A)\to (B,<_B,g_B)$ of partial automorphisms.  
 We say that $(B,<_B,f_B,g_B)$ extends $(A,<_A,f_A,g_A)$  and write
 \[(A,<_A,f_A,g_A)\preceq (B,<_B,f_B,g_B),\]
 if the inclusion $A\subseteq B$ induces an embedding of pairs of partial automorphisms.   Similarly we will define the  extension $(A,<_A)\preceq (B,<_B)$ between posets  and the extension  $(A,<_A,f_A)\preceq (B,<_B,f_B)$  between partial automorphisms---notice that posets and partial automorphisms can be viewed as special types of pairs of partial automorphisms  $(A,<_A,f_A,g_A)$ where either both $f_A,g_A$ or only $g_A$ are equal to the empty partial map $\emptyset$.

 Let $\mathcal{A}$ be the {\bf category of  embeddings between pairs of partial automorphisms}. By   \cite[Theorem 2.11]{KR}, the action of ${\rm Aut}(\mathcal{P})$ on  $\big({\rm Aut}(\mathcal{P})\big)^2$  by conjugation $g\cdot(f_1,f_2)\mapsto (gf_1g^{-1},gf_2g^{-1})$ has a dense orbit if and only if $\mathcal{A}$ has  the  \emph{joint embedding property} and by \cite[Theorem 6.2]{KR}  ${\rm Aut}(\mathcal{P})$ has a generic pair if and only if the category $\mathcal{A}$ has both the \emph{joint embedding property} and the \emph{weak amalgamation property}. Recall that a  category $\mathcal{C}$ has the {\bf Joint Embedding Property} ({\bf JEP}) if for any two objects $A,B\in\mathcal{C}$ 
there exists $C\in\mathcal{C}$ which embeds both $A$ and $B$.
A category $\mathcal{C}$ has the {\bf Weak Amalgamation Property} ({\bf WAP}) if for every object $A\in\mathcal{C}$ there is $\alpha \colon A\to \widetilde{A}$ in $\mathcal{C}$ so that for all $\beta \colon \widetilde{A}\to B$ and $\gamma\colon \widetilde{A}\to C$ in $\mathcal{C}$ there are
 $\beta'\colon B\to D$ and $\gamma'\colon C\to D$ in $\mathcal{C}$, so that $  \beta'\circ \beta \circ \alpha =\gamma' \circ \gamma \circ \alpha.$

 Notice that the continuous action of a Polish group with a dense orbit does not contain any  orbits which are neither meager nor comeager. Hence, by the previous paragraph, Theorem \ref{T:intro} follows from the following result.

\begin{theorem}\label{T}
The category $\mathcal{A}$ has JEP but does not have WAP.
\end{theorem}

The fact that  $\mathcal{A}$ has JEP is straightforward. Indeed, let $(A,<_A,f_A,g_A)\in \mathcal{A}$ and  $(B,<_B,f_B,g_B)\in \mathcal{A}$ and assume without loss of generality that $A \cap B=\emptyset$. Then both pairs of partial  automorphisms embed to the partial automorphism $(C,<_C,f_C,g_C)$, where  $C:=A\cup B$, $<_C$ is the partial order $<_A \cup \; \{(a,b)\mid a\in A, b\in B\} \; \cup <_B $, $f_C:=f_A\cup f_B$, and $g_C:=g_A\cup g_B$.

The rest of this paper is  devoted to showing that   $\mathcal{A}$  does not have WAP. We start with some definitions. 
Let $(A,<_A)\preceq (B,<_B)$ be an extension of posets and let $b\in B$. The {\bf quantifier free type  of $b$ over $A$ in $(B,<)$}, denoted by $\mathrm{qft}\big(b,A,(B,<_B)\big)$, is the collection of all formulas $\varphi(x)$ in some fixed variable $x$  of the form:
\[\varphi(x)\equiv (x<a) \; \;  \text{ or } \; \;  \varphi(x)\equiv (x>a) \; \; \text{ or }  \; \; \varphi(x)\equiv (x=a), \]
where $a\in A$ and $\varphi (b)$ holds in $(B,<_B)$. A {\bf quantifier free type $p(x)$ over $A$} is any collection of the form  $\mathrm{qft}\big(c,A,(C,<_C)\big)$, where $(A,<_A)\preceq (C,<_C)$ and $c\in C$.

If $p(x)$ is a quantifier free type over some poset $(A,<_A)$ and $\alpha\colon (A,<_A)\to (B,<_B)$ is an isomorphism, then we denote by $p[\alpha](x)$ the {\bf push forward of $p$ by $\alpha$}. This is the new quantifier free type over $(B,<_B)$ attained by setting for each $b\in B$: 
\begin{align*}
(x<b)\in p[\alpha](x)\iff (x<\alpha^{-1}(b))\in p(x)\\
(x>b)\in p[\alpha](x)\iff (x>\alpha^{-1}(b))\in p(x)\\
(x=b)\in p[\alpha](x)\iff (x=\alpha^{-1}(b))\in p(x)
\end{align*}

We record the following simple fact which we are going to use extensively below.

\begin{fact}\label{fact}
Let $(C,<_C,g)$  be a partial automorphism and let $c,d\in C$. If $p[g](x)=q[x]$, where  $p(x):=\mathrm{qft}\big(c,\mathrm{dom}(g),(C,<_C)\big)$ and  $q(x):=\mathrm{qft}\big(d,\mathrm{rng}(g),(C,<_C)\big)$,  then the map $g':=g\cup\{(c,d)\}$ induces a new partial automorphism $(C,<_C,g')$. 
\end{fact}

Let $(B,<_B,f_B)$ be a partial automorphism and let $a,b\in B$ with $a<_Bb$. We say that $f_B$ is {\bf free in $(a,b)$}, if whenever $(B,<_B)\preceq (C,<_C)$ and  $c_1,\ldots,c_{\ell} \in C$, with $a<_C c_1<_C \cdots <_C c_{\ell}<_C b$, then
$(C,<_C,f_C)$ is a partial automorphism, where 
\[f_C:= f_B \cup \{(a,c_1),(c_1,c_2),\ldots,(c_{\ell-1},c_{\ell})\}.\]

\begin{lemma}\label{L1}
Let $(A,<_A,f_A)$ be a partial automorphism and let $s\in A$ with $s<_A f_A(s)$. Then, there is an extension $(A,<_A,f_A)\preceq (B,<_B,f_B)$,  some $n\in\mathbb{N}$, and $a,b\in B$ with  $a<_Bb$, so that $f_B^n(s)=a$ and $f_B$ is free in $(a,b)$. 
\end{lemma}
\begin{proof}
Let $\{a_1,\ldots,a_k\}$ be an enumeration of $A$ and set $(A_0,<_0,f_0):=(A,<_A,f_A)$. Assume that for some $i<k$ we have defined  partial automorphisms  
\[(A_0,<_0,f_0)\preceq (A_1,<_1,f_1) \preceq \cdots \preceq (A_i,<_i,f_i),\] 
so that for some $n_i\in\mathbb{N}$  we have that 
\[f_{i}=f_A\cup\{(s,f_{i}(s)), (f_{i}(s),f^2_{i}(s)),\ldots, (f^{n_i-1}_{i}(s),f_i^{n_i}(s))\}.\]
We define a further extension   $(A_i,<_i,f_i) \preceq (A_{i+1},<_{i+1},f_{i+1})$ as follows:

\noindent {\bf Case 1.} If there is a partial automorphism $(D,<_D,f_D)$ which extends $(A_i,<_i,f_i)$, and some $m\geq n_i$ so that $f_D^m(s)$ is defined and $a_i<_D f_D^m(s)$, then let  $(A_{i+1}, <_{i+1},f_{i+1})$ be the unique partial automorphism with: 
\[(A_i,<_i,f_i) \preceq (A_{i+1},<_{i+1},f_{i+1})\preceq (D,<_D,f_D), \text{ so that }\]  $A_{i+i}=A_i\cup\{f^j_{D}(s))\mid 0\leq j\leq m\}$ and 
$f_{i+1}=f_{i}\cup \{(s,f_{D}(s)),\ldots,(f^{m-1}_{D}(s),f^{m}_{D}(s))\}$, and set  $n_{i+1}:=m$.

\noindent {\bf Case 2.} If Case 1 does not hold, but there is a partial automorphism $(D,<_D,f_D)$ which extends $(A_i,<_i,f_i)$, and some $m\geq n_i$ so that $f_D^m(s)$ is defined and $f_D^m(s)\perp a_i$, then let  $(A_{i+1}, <_{i+1},f_{i+1})$ be the unique partial automorphism with 
\[(A_i,<_i,f_i) \preceq (A_{i+1},<_{i+1},f_{i+1})\preceq (D,<_D,f_D), \text{ so that }\]  $A_{i+i}=A_i\cup\{f^j_{D}(s))\mid 0\leq j\leq m\}$ and 
$f_{i+1}=f_{i}\cup \{(s,f_{D}(s)),\ldots,(f^{m-1}_{D}(s),f^{m}_{D}(s))\}$, and set  $n_{i+1}:=m$.

\noindent {\bf Case 3.} If neither Case 1 nor Case 2 hold, then set $(A_{i+1},<_{i+1},f_{i+1}):=(A_i,<_i,f_i)$ and $n_{i+1}:=n_i$

\medskip

Given the resulting partial automorphism $(A_{k},<_{k},f_{k})$, set $n:=n_k+1$, and let $(E,<_E,f_{E})$ be any $\preceq$-minimal extension  of $(A_{k},<_{k},f_{k})$ which adds $f^{n_k}_k(s)$ into the domain of $f_E$. By minimality, $E=A_k\cup\{a\}$ for some  new point $a$ and $f_E=f_k\cup\{(f^{n_k}_k(s),a)\}$. Indeed, notice that if $a$ was not a new point, then it would have to come from $A$. But then,  $a=a_j$ for some $j\leq k$, and any further extension $(F,<_F,f_{F})$ of $(E,<_E,f_{E})$ which would include  $a$ in the domain of $f_{F}$, would activate Case 1 at the $j$-th stage of the above construction. As a consequence, by the end of the construction, we would have  $a<_k f_k^{n_k}(s)$, and since the orbit of $s$ under any extension of $f_A$ is increasing that would contradict  that  $f^n_E(s)=a$.

We may finally extend  $(E,<_E,f_{E})\preceq (B,<_B,f_{B})$ to the desired  partial automorphism $(B,<_B,f_{B})$ which is defined by: $B:= E\cup\{b\}$ where $b\not \in E$; $a<_B b$ and $\mathrm{qft}\big(b,E \setminus \{a\} ,(B,<_B)\big)=\mathrm{qft}\big(a, E \setminus \{a\} ,(B,<_B)\big)$; and $f_B:=f_E$. It is easy to check that $(B,<_B)$ is a poset and therefore $(B,<_B,f_{B})$ is a partial automorphism extending $(A,<_A,f_{A})$ and  $f_B^n(s)=a<_Bb$.

We are left to show that $f_B$ is free in $(a,b)$. 
Let  $(B,<_B)\preceq (C,<_C)$ with $a<_C c_1<_C \cdots <_C c_{\ell}<_C b$ for some $c_1,\ldots,c_{\ell} \in C$.  The fact that $(C,<_C,f_C)$, with
\[f_C:= f_B \cup \{(a,c_1),(c_1,c_2),\ldots,(c_{\ell-1},c_{\ell})\}.\]
is a partial automorphism follows by a straightforward induction based on Fact \ref{fact} and the following claim.

\begin{claim}
Let $c\in C$ with $a<_Cc<_Cb$.  Then  we have that:
\begin{enumerate}
%\item if $p(x):=\mathrm{qft}\big(a,\mathrm{dom}(f_B),(C,<_C)\big)\big)$, then $p(x)=\mathrm{qft}\big(c,\mathrm{dom}(f_B),(C,<_C)\big)$;
\item $\mathrm{qft}\big(c,B\setminus\{a\},(C,<_C)\big)\big)=\mathrm{qft}\big(a,B\setminus\{a\},(C,<_C)\big)\big)$;
\item if $p(x):=\mathrm{qft}\big(c,\mathrm{dom}(f_B),(C,<_C)\big)$ and $q(x):=\mathrm{qft}\big(c,\mathrm{rng}(f_B),(C,<_C)\big)$, then \[q(x)=p[f_B](x).\]
\end{enumerate} 
\end{claim}
\begin{proof}[Proof of Claim]
For simplicity we will denote  $<_C$ by $<$ throughout this proof.
\medskip{}

\noindent (1)  follows from transitivity of $<$, since  $a<c<b$ and  $\mathrm{qft}\big(b,E \setminus \{a\} ,(B,<_B)\big)=\mathrm{qft}\big(a, E \setminus \{a\} ,(B,<_B)\big)$.\medskip{}

\noindent For  (2), we need to show that for every $w\in \mathrm{dom}(f_B)$ we have that:
\begin{enumerate}
\item[(i)]  $c=f_B(w)$ if and only if $c=w$; 
\item[(ii)]  $c< f_B(w)$ if and only if $c<w$;
\item[(iii)]  $c> f_B(w)$ if and only if $c>w$.
\end{enumerate} 
By part (1) of the claim, and since $\mathrm{dom}(f_B)  \subseteq B\setminus \{a\}$, it is enough to show that  for every $w\in \mathrm{dom}(f_B)  \subseteq B\setminus \{a\}$ we have that: 
\begin{enumerate}
\item[(i)$'$]  $c= f_B(w)$ if and only if $a=w$;
\item[(ii)$'$]  $c< f_B(w)$ if and only if $a<w$;
\item[(iii)$'$]  $c> f_B(w)$ if and only if $a>w$.
\end{enumerate} 
Notice that (i)$'$ is trivially true since $c\not\in \mathrm{rng}(f_B)$. Moreover, for $w=f^{-1}_B(a)=f^{n-1}_B(s)$, both (ii)$'$ and (iii)$'$ are satisfied
since $w<f_B(w)<c$.
If now $w\in \mathrm{dom}(f_B)\setminus\{f^{-1}_B(a)\}$, then  $f_B(w)\in B\setminus \{a\}$. Hence, by a second application of part (1) of the claim we are left to show that for all   $w\in \mathrm{dom}(f_B)\setminus\{f^{-1}_B(a)\}$ we have that:
\begin{enumerate}
\item[(ii)$''$]  $a< f_B(w)$ if and only if $a<w$;
\item[(iii)$''$]  $a> f_B(w)$ if and only if $a>w$.
\end{enumerate} 
This statement is clear if $w= f_B^{l}(s)$ for some $l<n-1$, since $f_B$ is increasing in the $f_B$-orbit of $s$ and therefore $w<f_B(w)<a=f_B^n(s)$.  We may therefore assume that both $w\in A$ and $f_B(w)\in A$. Hence, $w=a_i$  and $f_B(w)=a_j$ for some $i,j\leq k$. We  consider the following three cases:

\noindent $\bullet$ If $w>a$, then $w>a>f_B^{-1}(a)$, and hence $f_B(w)>a$.

\noindent $\bullet$ If $w<a$, then any extension $g$ of $f_B$ with $a\in \mathrm{dom}(g)$ satisfies  $g(w)<g(a)$, i.e.,
\[a_j=f_B(w)=g(w)<g(a)=g(f^n_B(s)).\]
But any such extension $g$ is, in particular, an extension of $f_{j-1}$ in the construction above. Hence, by Case 1 at the $j$-th step of the  construction we  already have that $a_j<f_{j}^{n_j}(s)$.  But then,  $f_B(w)=a_j<f_{j}^{n_j}(s) <a$.

\noindent $\bullet$ If  $w\perp a$, then by  Case 2 at the $i$-th step of the above construction we have that $w\perp f_B^{n_i}(s)$. Since $f^{n_i}_B(s)\leq f^{-1}_B(a)<a$ we have that $w\perp f^{-1}_B(a)$. Hence, $f_B (w)\perp a$.
\end{proof}

Let  $p_0(x):=\mathrm{qft}\big(a,\mathrm{dom}(f_B),(C,<_C)\big)$ and  $q_1(x):=\mathrm{qft}\big(c_1,\mathrm{rng}(f_B), (C,<_C)\big)$. Using both (1),(2) of the last claim we have that  $q_1(x)=p_0[f_B](x)$. By Fact \ref{fact}  we have that $(C,<_C,f_B\cup\{(a,c_1)\})$ is a partial automorphism. 
By induction we   show that $(C,<_C,f_C)$ is a partial automorphism. The inductive step is as follows. Take $t$ with $0<t\leq \ell$.  
By (1) and (2) of the last claim we have that   $q'_{t}(x)=p'_{t-1}[f_B](x)$, where
\[q'_{t}(x)=\mathrm{qft}\big(c_t,\mathrm{rng}(f_B), (C,<_C)\big) \; \text{ and } \; p'_{t-1}(x)=\mathrm{qft}\big(c_{t-1},\mathrm{dom}(f_B),(C,<_C)\big).\]
 Since $a<_Cc_1<_C\cdots<_Cc_{t-1}<_Cc_t$, we also have that the quantifier-free type
\[\mathrm{qft}\big(c_t,\{c_1,\ldots,c_{t-1}\}, (C,<_C)\big)\]
is the push-forward of the the quantifier-free type
\[\mathrm{qft}\big(c_{t-1},\{a,c_1,\ldots,c_{t-2}\}, (C,<_C)\big),\]
under the partial automorphism $\{(a,c_1),(c_1,c_2),\ldots,(c_{t-2},c_{t-1})\}$.
Combining these two facts, we have that  the quantifier-free type
\[q_t(x):=\mathrm{qft}\big(c_t,\mathrm{rng}(f_B)\cup\{c_1,\ldots,c_{t-1}\}, (C,<_C)\big)\]
is the push forward of  the quantifier-free type
 \[p_{t-1}(x):=\mathrm{qft}\big(c_{t-1},\mathrm{dom}(f_B)\cup\{a,c_1,\ldots,c_{t-2}\},(C,<_C)\big)\] 
under  the map $f_B\cup \{(a,c_1),(c_1,c_2),\ldots,(c_{t-2},c_{t-1})\}$. By Fact \ref{fact}, if the latter is a partial automorphism, then so is $f_B\cup \{(a,c_1),(c_1,c_2),\ldots,(c_{t-2},c_{t-1}),(c_{t-1},c_{t})\}$. 

Hence, $(C,<_C,f_C)$ is indeed a partial automorphism.

\end{proof}

We can now conclude with the proof of Theorem \ref{T} which implies Theorem \ref{T:intro}.

\begin{proof}[Proof of Theorem \ref{T}]

Consider the pair $(A,<_A,f_A,g_A)$ of partial automorphisms with:
\begin{enumerate}
\item  $A=\{x,y,z\}$ with $x<_Ay<_Az$;
\item $f_A=\{(x,y), (z,z)\}$ and $g_A=\{(x,z)\}$.
\end{enumerate}
Let  $(\widetilde{A},<_{\widetilde{A}},f_{\widetilde{A}},g_{\widetilde{A}})$  be any extension of $(A,<_A,f_A,g_A)$. We will define two further extensions $(A_1,<_1,f_1,g_1)$ and $(A_2,<_2,f_2,g_2)$ of $(\widetilde{A},<_{\widetilde{A}},f_{\widetilde{A}},g_{\widetilde{A}})$ which do not amalgamate over  $(A,<_A,f_A,g_A)$. We start by defining some initial extension $(A_0,<_0,f_0,g_0)$ of  $(\widetilde{A},<_{\widetilde{A}},f_{\widetilde{A}},g_{\widetilde{A}})$   which will be contained in both $(A_1,<_1,f_1,g_1)$ and $(A_2,<_2,f_2,g_2)$.

Apply Lemma~\ref{L1} to  $(\widetilde{A},<_{\widetilde{A}},f_{\widetilde{A}})$ for $s:=x$  to get some extension $(B,<_B,f_B)$, some $n\in\mathbb{N}$, and some $a_f,b_f\in B$ with $a_f<_B b_f$, so that $(f_B)^n(x)=a_f$ and $f_B$ is free in $(a_f,b_f)$. By extending further $(B,<_B)$ if necessary we can assume without loss of generality that there is some 
$c\in B$ with $a_f<_Bc<_Bb_f$.  Notice that since $f_A(z)=z$ and $x<_B z$ we have that $f^n_B(x)<_Bz$ for all $n\in\mathbb{N}$. 
Since by the construction in the proof of  Lemma \ref{L1} we have that 
\[\mathrm{qft}\big(a_f, B \setminus \{a_f,b_f\} ,(B,<_B)\big)=\mathrm{qft}\big(b_f,B \setminus \{a_f,b_f\} ,(B,<_B)\big),\]
 it follows that $b_f<_B z$. 
 Since $x<_Bc<_Bz$ and $g_A(x)=z$,  any extension $(E,<_E,g)$ of $(B,<_B,g_A)$ with $c\in \mathrm{dom}(g)$ satisfies $c<_Eg(c)$. 
 Hence, by Lemma~\ref{L1} applied to an arbitrarily chosen such extension  $(E,<_E,g)$ with $s:=c$ we get a further extension $(B,<_B,g_A)\preceq (C,<_C,g_C)$,  some $m\in\mathbb{N}$, and  $a_g,b_g\in C$ with $a_g<_Cb_g$, so that $g_C^m(c)=a_g$ and $g_C$ is free in $(a_g,b_g)$.

Set $c_0:=c, c_1:=g_C(c),\cdots, c_m :=g_C^m(c)$. By extending $C$ further if necessary, we may assume without loss of generality that for every $i$ with $0\leq i \leq m$ there is a point $d_i\in C$ with   $d_i\not\in \mathrm{dom}(g_C)$ and $c_i<_Cd_i$, and so that:
\[\mathrm{qft}\big(d_i,C\setminus \{c_i,d_i\},(C,<_C)\big)=\mathrm{qft}\big(c_i,C\setminus \{c_i,d_i\},(C,<_C)\big).\]
Set $(A_0,<_0):=(C,<_C)$. By a simple induction using Fact \ref{fact}  we have that  $(A_0,<_0,g_0)$ with  $g_0:=g_C\cup \{(d,d_1), (d_1,d_2),\ldots,(d_{m-1},d_m)\}$, is a partial automorphism. Since $f_B$ is free in $(a_f,b_f)$ and $a_f<_0c<_0d<_0b_f$, we similarly have that  $(A_0,<_0,f_0)$ with  $f_0:=f'\cup \{(a_f,c),(c,d)\}$ is also a partial automorphism. 
We therefore have a pair $(A_0,<_0,f_0,g_0)$ of partial automorphisms.

We may now define two further extensions of $(A_0,<_0,f_0,g_0)$ which do not amalgamate over $(A,<_A,f_A,g_A)$ as follows:  $(A_1,<_1,f_1,g_1)$ is defined by setting $A_1:=A_0, f_1:=f_0, g_1:= g_0\cup \{(g^m_0(c),d_m)\}$; and $(A_2,<_2,f_2,g_2)$ is defined by $A_2:=A_0\cup\{e\}, f_2:=f_0, g_2:= g_0\cup \{(g^m_0(c),e),(e,d_m)\}$, where $e$ is a new point with $g^m_0(c)<_2e<_2d_m$. Since  $a_g<_2g^m_0(c)<_2e<_2d_m<_2b_g$ and $g_0$ is free in $(a_g,b_g)$, we have that both  $(A_1,<_1,g_1)$  and $(A_2,<_2,g_2)$ are  partial automorphisms.

Finally, to see that  $(A_1,<_1,f_1,g_1)$ and  $(A_2,<_2,f_2,g_2)$ cannot be amalgamated over $(A,<_A,f_A,g_A)$, notice that in any such amalgam, the copy of $x$ in $A_1$ has to be identified with the copy of $x$ in $A_2$. As a consequence, in such amalgam, the point $g_1^t\circ f_1^s(x)$ from $A_1$ has to be identified with the point  $g_2^t\circ f_2^s(x)$ from $A_2$,  for all $s,t\in\mathbb{N}$ for which both $g_1^t\circ f_1^s(x)$ and $g_2^t\circ f_2^s(x)$ are defined.  However, by the above construction we have that  $(g_1^{m+1} \circ f_1^{n+1})(x)= (g_1^{m} \circ f_1^{n+2})(x)$ in $A_1$ and $(g_2^{m+1} \circ f_2^{n+1})(x)<_2 (g_2^{m} \circ f_2^{n+2})(x)$.
\end{proof}

\end{document}